\documentclass[12pt]{amsart}
\usepackage{amsmath,amsthm,amssymb,amscd,url,enumerate,mathtools}
\usepackage{graphicx}
\usepackage[margin=1in]{geometry}
\usepackage{afterpage}
\usepackage{tikz}
\usepackage{todonotes}
\usepackage[all]{xy}
\usepackage[colorlinks=true,citecolor={blue},linkcolor = {purple}]{hyperref} 
\usepackage{tabularx}

\usepackage{caption}

\usepackage[square]{natbib}
\setcitestyle{numbers}

\newcommand{\TITLE}{Radical Dynamical Monogenicity}
\newcommand{\TITLERUNNING}{}

\theoremstyle{plain}
\newtheorem{theorem}{Theorem}

\newtheorem{proposition}[theorem]{Proposition}
\newtheorem{lemma}[theorem]{Lemma}
\newtheorem{corollary}[theorem]{Corollary}

\theoremstyle{definition}

\theoremstyle{remark}

\newtheorem{example}[theorem]{Example}
\newtheorem{question}[theorem]{Question}

\numberwithin{theorem}{section}

  {\end{list}}

  {\end{list}}

\newcommand{\tightoverset}[2]{%
  \mathop{#2}\limits^{\vbox to -.5ex{\kern-1.05ex\hbox{$#1$}\vss}}}

\numberwithin{equation}{section} 

\newcommand{\gm}{{\mathfrak{m}}}
\newcommand{\gn}{{\mathfrak{n}}}
\newcommand{\gp}{{\mathfrak{p}}}

\newcommand{\gP}{{\mathfrak{P}}}

\newcommand{\gl}{{\mathfrak{l}}}

\def\Ocal{{\mathcal O}}

\newcommand{\FF}{\mathbb{F}}

\newcommand{\QQ}{\mathbb{Q}}

\newcommand{\ZZ}{\mathbb{Z}}

\newcommand{\tensor}{\otimes}
\newcommand{\dnd}{\nmid}

\newcommand{\ol}[1]{\overline{#1}}

\newcommand{\Disc}{\operatorname{Disc}}
\newcommand{\ind}{\operatorname{ind}}

\newcommand{\Norm}{\operatorname{Norm}}
\newcommand{\red}{\operatorname{red}}

\title[\TITLERUNNING]{\TITLE}

\author[Hanson Smith]{Hanson Smith}
\address{Department of Mathematics\\
California State University San Marcos\\
333 S. Twin Oaks Valley Rd.\\
San Marcos, CA 92096}
\email{hsmith@csusm.edu}


\begin{document}

\sloppy 


\baselineskip=17pt


\begin{abstract}
Let $a$ be an integer and $p$ a prime so that $f(x)=x^p-a$ is irreducible. Write $f^n(x)$ to indicate the $n$-fold composition of $f(x)$ with itself. We study the monogenicity of number fields defined by roots of $f^n(x)$ and give necessary and sufficient conditions for a root of $f^n(x)$ to yield a power integral basis for each $n\geq 1$.
\end{abstract}

\maketitle

\section{Results and Previous Work}
Let $L$ be a number field. We will denote the ring of integers by $\Ocal_L$. Suppose $M$ is a finite extension of $L$. If $\Ocal_M=\Ocal_L[\alpha]$ for some $\alpha\in \Ocal_M$, then we say \emph{$M$ is monogenic over $L$} or $\Ocal_M$ has a \emph{power $\Ocal_L$-integral basis.} In this case, we call $\alpha$ a \textit{monogenerator}. If $f(x)$ is the minimal polynomial of $\alpha$ over $L$, then we also call $f(x)$ \textit{monogenic}. 
When $L$ is $\QQ$ we will simply say $M$ is \emph{monogenic} or $\Ocal_M$ has a \emph{power integral basis.} 

If $\gp$ is a prime of $\Ocal_L$, then we write $(\Ocal_L)_\gp$ for the localization of $\Ocal_L$ at $\gp$. Given an $\Ocal_M$-order $R$, we say $R$ is \textit{$\gp$-maximal} if $R\tensor (\Ocal_L)_\gp \cong \Ocal_M\tensor (\Ocal_L)_\gp$. Indeed, if $R$ is $\gp$-maximal for each prime $\gp$ of $\Ocal_L$, the $R$ is the maximal order $\Ocal_M$.

If $f(x)\in\ZZ[x]$ is a polynomial, then we write 
\[f^n=\underbrace{f\circ f\circ \cdots \circ f}_{n\text{-times}}.\]
This paper investigates the monogenicity of extensions obtained by adjoining a root of $f^n(x)$ where $f(x)=x^p-a$ is an irreducible integer polynomial with $p$ a prime number. In particular, we prove the following.


\begin{theorem}\label{Thm: Main}
Write $\alpha_n$ for a root of $f^n(x)$, and let $K_n$ denote $\QQ(\alpha_n)$. Fix a natural number $N\geq 1$. The ring of integers $\Ocal_{K_n}$ is $\ZZ[\alpha_n]$ for all $n\leq N$ if and only if $a^p\not \equiv a \bmod p^2$ and $f^n(0)$ is squarefree for all $n\leq N$.
\end{theorem}


\subsection{Previous Work}

The literature regarding monogenic fields is vast. A recent text on monogenicity that focuses on using index form equations is Ga\'al's book \cite{GaalsBook}. 
Another modern resource is Evertse and Gy\H{o}ry's book \cite{EvertseGyoryBook}. Radical extensions\footnote{These extensions are also called \textit{pure extensions}.} also have a very extensive body of previous work. We will not undertake a general survey of the literature involving radical extensions and monogenicity here; the curious reader can see \cite{SmithRadical} for this.


Our investigation is inspired by \cite{Ruofan} where Li studies iterates of $f(x)=x^2-a$. Li uses novel arguments to give necessary and sufficient conditions for a root of $f^n(x)$ to be a monogenerator for all $n\leq N$, a fixed positive integer. We extend Li's results to $x^p-a$ with $p$ an odd prime.

Other authors have studied the monogenicity of iterates of polynomials defining radical extensions. In \cite{JonesStable}, 
Jones considers $h(x)=(x-t)^m+t$ for $m\geq 2$ and $t\geq 1$. This polynomial yields the same extension as $g(x)=x^m+t$; however, the fields defined by roots of $h^n(x)$ and $g^n(x)$ differ for $n\geq 2$. Indeed, $h^n(x)$ yields the same number field as $x^{m^n}+t$ for all $n\geq 1$, but $g^n(x)$ does not. 


Given the simple shape of radical polynomials and the fact that $x^d+c$, up to a change of variables, represents all polynomials of degree $d$ with exactly one finite cricital point, it is not surprising that a number of authors have studied the dynamics of these polynomials. We mention a few papers which are adjacent to our note. In \cite{FiniteIndexGaloisGroups}, the authors investigate the Galois groups of iterates of $x^q+c$. The work in \cite{FeinSchacher} studies the basic properties of iterates and composites of polynomials, finding properties such as irreducibility, separability, complete splitting, and solubility by radicals are not necessarily preserved. The specific case of irreducibility of iterates of a radical polynomial $x^d+c$ is considered in \cite{DanielsonFein}. They investigate irreducibility over a wide class of fields, but of primary interest to us will be their Corollary 5, which shows that if $h(x)=x^d+c$ is irreducible in $\ZZ[x]$, then so are its iterates.


\section{Background: The Montes Algorithm and Ore's Theorems}

The Montes algorithm is a powerful $p$-adic factorization algorithm that is based on and extends the pioneering work of {\O}ystein Ore \cite{Ore}. We do not need the full strength of the Montes algorithm here. In fact, though we use the notation and setup of the general implementation, we will only make use of the aspects developed by Ore. The following is a brief summary of the tools we will utilize; for the complete development of the Montes algorithm, see \cite{GMN}. Our notation will roughly follow \cite{ElFadilMontesNart}, which gives a more extensive summary than we undertake here.

Let $p$ be an integral prime, $K$ a number field with ring of integers $\Ocal_K$, and $\gp$ a prime of $K$ above $p$. Write $K_\gp$ to denote the completion of $K$ at $\gp$. Suppose we have a monic, irreducible polynomial $f(x)\in \Ocal_K[x]$. We extend the standard $\gp$-adic valuation to $\Ocal_K[x]$ by defining the $\gp$-adic valuation of $f(x) = a_n x^n + \cdots + a_1 x + a_0 \in \Ocal_K[x]$ to be 
	\[ v_\gp\big(f(x)\big) = \min_{0 \leq i \leq n} \big( v_\gp(a_i) \big). \]
This is often called the \textit{Gauss valuation.}
If $\phi(x), f(x) \in \Ocal_K[x]$ are such that $\deg \phi \leq \deg f$, then we can write
    	\[f(x)=\sum_{i=0}^k a_i(x)\phi(x)^i,\]
for some $k$, where each $a_i(x) \in \Ocal_K[x]$ has degree less than $\deg \phi$. We call the above expression the \emph{$\phi$-adic development} of $f(x)$. We associate to the $\phi$-adic development of $f(x)$ an open Newton polygon by taking the lower convex hull 
of the integer lattice points $\big(i,v_p(a_i(x))\big)$. The sides of the Newton polygon with negative slope are the \emph{principal $\phi$-polygon}. 
The positive integer lattice points on or under the principal $\phi$-polygon contain a wealth of arithmetic information. We denote the number of such lattice points by $\ind_\phi(f)$.

Write $k_\gp$ for the field $\Ocal_K/\gp$, and let $\ol{f(x)}$ be the image of $f(x)$ in $k_\gp[x]$. It will often be the case that we develop $f(x)$ with respect to an irreducible factor $\phi(x)$ of $\ol{f(x)}$. In this situation, we will want to consider the extension of $k_\gp$ obtained by adjoining a root of $\phi(x)$. We denote this finite field by $k_{\gp,\phi}$. We associate to each side of the principal $\phi$-polygon a polynomial in $k_{\gp,\phi}[y]$. Suppose $S$ is a side of the principal $\phi$-polygon with initial vertex $\big(s,v_\gp(a_s(x))\big)$, terminal vertex $\big(k,v_\gp(a_k(x))\big)$, and slope $-\frac{h}{e}$ written in lowest terms. Define the length of the side to be $l(S)=k-s$ and the degree to be $d\coloneqq\frac{l(S)}{e}$. Let $\red:\Ocal_K[x]\to k_{\gp,\phi}$ denote the homomorphism obtained by quotienting by the ideal $\big(\gp,\phi(x)\big)$.
For each $i$ in the range $b\leq i\leq k$, we define the residual coefficient to be
\[c_i=\left\{
\begin{array}{ll}
0 \text{ if }  \big(i,v_\gp(a_i(x))\big)  \text{ lies strictly above } S  \text{ or } v_\gp(a_i(x))=\infty,\\
\red\left(\frac{a_i(x)}{\pi^{v_\gp(a_i(x))}}\right)  \text{ if }  \big(i,v_\gp(a_i(x))\big) \text{ lies on } S.
\end{array}
\right.\]
Finally, the \emph{residual polynomial} of the side $S$ is the polynomial
\[R_S(y)=c_s+c_{s+e}y+\cdots +c_{s+(d-1)e}y^{d-1}+c_{s+de}y^d\in k_{\gp,\phi}[y].\]
Notice, that $c_s$ and $c_{s+de}$ are always nonzero since they are the initial and terminal vertices, respectively, of the side $S$.

With all of these definitions in hand, we package everything into two theorems that encapsulate how we will employ the Montes algorithm. The first focuses on the indices of monogenic orders.

\begin{theorem}[Ore's theorem of the index]\label{Thmofindex}
Let $f(x)\in \Ocal_K[x]$ be a monic irreducible polynomial and let $\alpha$ be a root. Choose monic polynomials $\phi_1,\dots, \phi_s \in \Ocal_K[x]$ whose reductions modulo $\gp$ are exactly the distinct irreducible factors of $\overline{f(x)} \in k_\gp[x]$. Then, 
\[v_p([\Ocal_{K(\alpha)}:\Ocal_K[\alpha]])\geq \ind_{\phi_1}(f)+\cdots + \ind_{\phi_s}(f).\]
Further, equality holds if, for every $\phi_i$, each side of the principal $\phi_i$-polygon has a separable residual polynomial.
\end{theorem}
One should consult \cite{JhorarKhanduja} for a generalization of the above theorem. For our applications, we will employ the following equivalence.

\begin{corollary}\label{iffCor}
The prime $\gp$ does not divide $[\Ocal_{K(\alpha)}:\Ocal_K[\alpha]]$ if and only if $\ind_{\phi_i}(f)=0$ for all $i$. In this case each principal $\phi_i$-polygon is one-sided.
\end{corollary}


The second theorem we state connects prime splitting and polynomial factorization. The ``three dissections" that we will outline below are due to Ore, and the full Montes algorithm is an extension of this. Our statement loosely follows Theorem 1.7 of \cite{ElFadilMontesNart}.

\begin{theorem}\label{Thm: Ore}[Ore's Three Dissections]
Let $f(x)\in \Ocal_K[x]$ be a monic irreducible polynomial and let $\alpha$ be a root. Suppose
\[\ol{f(x)}=\phi_1(x)^{r_1}\cdots \phi_s(x)^{r_s}.\]
is a factorization into irreducibles in $k_\gp[x]$. Hensel's lemma shows $\phi_i(x)^{r_i}$ corresponds to a factor of $f(x)$ in $K_\gp[x]$ and hence to a factor $\gm_i$ of $\gp$ in $K(\alpha)$. 

Choosing a lift and abusing notation, suppose the principal $\phi_i$-polygon has sides $S_1,\dots, S_g$. Each side of this polygon corresponds to a distinct factor of $\gm_i$. 

Write $\gn_j$ for the factor of $\gm_i$ corresponding to the side $S_j$. Suppose $S_j$ has slope $-\frac{h}{e}$. If the residual polynomial $R_{S_j}(y)$ is separable, then the prime factorization of $\gn_j$ mirrors the factorization of $R_{S_j}(y)$ in $k_{\gp,\phi_i}[y]$, but every factor of $R_{S_j}(y)$ will have an exponent of $e$. In other words,
\[\text{if } R_{S_j}(y)=\gamma_1(y)\dots\gamma_k(y) \ \text{ in }  \ k_{\gp,\phi_i}[y], \ \text{ then } \ \gn_j=\gP_1^{e}\cdots \gP_k^{e} \ \text{ in } \ K(\alpha)\]
with $\deg(\gamma_m)$ equaling the residue class degree of $\gP_m$ for each $1\leq m\leq k$.
In the case where $R_{S_j}(y)$ is not separable, further developments are required to factor $\gp$. 
\end{theorem}



\section{General Radical Extensions and Lemmas}\label{genradsection}

In this section we recall and restate Theorem 6.1 of \cite{SmithRadical}. This Theorem will be a key lemma in our study of the monogenicity of $f^n(x)$. The notation in this section follows \cite{SmithRadical} and should not be conflated with our notation elsewhere.

Consider an arbitrary number field $L$ and an element $\alpha\in \Ocal_L$ such that $x^n-\alpha$ is irreducible over $L$. To avoid trivialities, we assume $n\geq 2$. One computes
\begin{equation}\label{discf}
\Disc\big(x^n-\alpha\big)=\left(-1\right)^\frac{n^2-n}{2}n^n(-\alpha)^{n-1}.
\end{equation}
The primes dividing this discriminant are the primes for which $\Ocal_L\left[\sqrt[n]{\alpha}\right]$ may fail to be maximal.

For a prime ideal $\gp$ of $\Ocal_L$ dividing $n$, we write $p$ for the residue characteristic and $f$ for the residue class degree. If $\gp$ divides $n$, we factor $n=p^em$ with $\gcd(m,p)=1$. 
Let $\varepsilon$ be congruent to $e$ modulo $f$ with $1\leq \varepsilon\leq f$, and define $\beta$ to be $\alpha$ to the power $p^{f-\varepsilon}$:
\[\beta\coloneqq \alpha^{p^{f-\varepsilon}} .\] 
By construction $\beta$ is the $p^e$-th root of $\alpha$ modulo $\gp$. 

\begin{theorem}[Theorem 6.1 of \cite{SmithRadical}]\label{Thm: SmithRadical}
If $\gp\mid \alpha$, then the order $\Ocal_L\left[\sqrt[n]{\alpha}\right]$ is $\gp$-maximal if and only if $v_\gp( \alpha)=1$. If $\gp\mid n$ and $\gp\nmid \alpha$, then the order $\Ocal_L\left[\sqrt[n]{\alpha}\right]$ is $\gp$-maximal if and only if \[v_\gp\left(\alpha-\beta^{p^e} \right)=1. \]
If $\gp\dnd \alpha n$, then $\Ocal_L\left[\sqrt[n]{\alpha}\right]$ is $\gp$-maximal.
\end{theorem}

When the base field is $\QQ$, one can use results in \cite{Westlund}, \cite{Alden}, or \cite{JKSTrinomialIndex}. However, for our work on iterates, we will employ the generality of Theorem \ref{Thm: SmithRadical}. We will also have the following lemmas.

\begin{lemma}\label{Lem: ExpDoesntMatter}
Let $p$ be a prime and let $a\in \ZZ$, then 
\[v_p\left(a^p-a\right) = v_p\left(a^{p^m}-a\right)\]
for every $m>0$.
\end{lemma}

\begin{proof}
Write $w=v_p(a^p-a)$. If $p\mid a$, then this is clear. Suppose $p\nmid a$. It suffices to show that 
\[v_p\left(a^{p-1}-1\right) = v_p\left(a^{p^m-1}-1\right)\]
The smallest of Fermat's theorems tells us that the base-$p$ expansion of $a^{p-1}$ has the form 
\[a^{p-1}=1+a_wp^w+(\text{higher powers of }p)\]
where $0< a_w <p$. Clearly, 
\[v_p\left(a^{p-1}-1\right)=v_p\left(a_wp^w+(\text{ higher powers of }p)\right)=w.\]
Note $p^m-1=(p-1)(p^{m-1}+p^{m-2}+\cdots+p+1)$, so
\begin{align*}
a^{p^m-1} &= \left( a^{p-1}\right)^{p^{m-1}+p^{m-2}+\cdots+p+1}\\
&=\big(1+a_wp^w+(\text{higher powers of }p)\big)^{p^{m-1}+p^{m-2}+\cdots+p+1}\\
&=1 +\left(p^{m-1}+p^{m-2}+\cdots+p+1\right) a_wp^w + (\text{higher powers of }p).
\end{align*}
We can now see that 
\[v_p\left(a^{p^m-1}-1\right)=v_p\left(a_wp^w+(\text{higher powers of }p)\right)=w. \qedhere \]
\end{proof}


\begin{lemma}\label{Lem: Val Norms}
Let $M_\gP/L_\gp$ be a finite extension of local fields that is totally ramified of degree $n$, where we write $\gP$ and $\gp$ for the the maximal ideals of the respective valuation rings. Write $\pi_\gP$ and $\pi_\gp$ for the uniformizers, so $v_\gP\big(\pi_\gp\big)=n$. If $\alpha$ is in the valuation ring of $M_\gp$, then 
\[v_\gP(\alpha)=v_\gp\left(\Norm_{M_\gP/L_\gp}(\alpha)\right).\] 
\end{lemma}

\begin{proof}
Write $m_{\pi_\gp}(x)$ for the minimal polynomial of $\pi_\gP$ over $L_\gp$, and note that $m_{\pi_\gp}(x)$ is $\gp$-Eisenstein. Thus,
\[v_\gp\big(\Norm(\pi_\gP)\big) = v_\gp\big(m_{\pi_\gP}(0)\big) = 1.\]

If $v_\gP(\alpha)=j$, then $\alpha=u\pi_\gP^j$ with $u$ a unit. Hence
\[v_\gp\big(\Norm_{M_\gP/L_\gp}(\alpha)\big)=v_\gp\Big(\Norm_{M_\gP/L_\gp}(u)\Norm_{M_\gP/L_\gp}\big(\pi_\gP^j\big)\Big)=jv_\gp\big(\Norm_{M_\gP/L_\gp}(\pi_\gP)\big)=j. \qedhere \]
\end{proof}


\section{Necessary and Sufficient Criteria for the Monogenicity of $f^n(x)$}

Recall, we take $f(x)=x^p-a$ to be irreducible, and we write
\[f^n = \underbrace{f\circ f \circ \cdots \circ f}_{n\text{-times}}.\]
The following special case of a result of Danielson and Fein shows that we do not need to be concerned with irreducibility after the first iterate. 
\begin{proposition}[Corollary 5 of \cite{DanielsonFein}]
If $f(x)=x^p-a$ is irreducible in $\ZZ[x]$, then $f^n(x)$ is irreducible in $\ZZ[x]$ for all $n>0$.
\end{proposition}

For shorthand, we write $\alpha_n$ for a root of $f^n(x)$, so $\alpha_1=\zeta_p^i\sqrt[p]{a}$ for some $i$. Further, write $K_n$ for $\QQ(\alpha_n)$. 
Motivated by \cite{Ruofan}, we wish to show the following:\\

\noindent\textbf{Theorem \ref{Thm: Main}.} 
\textit{Fix a natural number $N\geq 1$, then $\Ocal_{K_n}=\ZZ[\alpha_n]$ for all $n\leq N$ if and only if $a^p\not \equiv a \bmod p^2$ and $f^n(0)$ is squarefree for all $n\leq N$.}\\

We will establish this theorem via two lemmas. The first considers the prime $p$ while the second considers primes $\ell\neq p$.

\begin{lemma}\label{Lem: pmaximal}
The order $\ZZ[\alpha_n]$ is $p$-maximal for all $n>0$ if and only if $a^p\not\equiv a\bmod p^2$. When this is the case, $p$ is totally ramified in $K_n$. Moreover, if $a^p\equiv a\bmod p^2$, then $\ZZ[\alpha_1]$ is not $p$-maximal.
\end{lemma}

\begin{proof}
First, we consider the case where $p\mid a$. In this case the $x$-adic development shows $\ZZ[\alpha_1]$ is $p$-maximal if and only if $a^p\not\equiv a\bmod p^2$. Thus this condition is necessary. Supposing now that $a^p\not\equiv a\bmod p^2$, a straightforward computation shows that $f^n(x)\equiv x^{p^n}\bmod p$ and $v_p(f^n(0))=v_p(a)$. Considering the $x$-adic development again, $\ZZ[\alpha_n]$ is $p$-maximal if and only if $a^{p^n}\not \equiv a \bmod p^2$ which holds if and only if $a^p\not \equiv a \bmod p^2$ by Lemma \ref{Lem: ExpDoesntMatter}. The slope of the principal $x$-polygon is $\frac{1}{p^n}$ in each of these cases, so $p$ is totally ramified. To be more succinct, in this case, each $f^n(x)$ is $p$-Eisenstein if and only if $f(x)$ is $p$-Eisenstein. 

Now we consider the more subtle case where $p\nmid a$, and we proceed by using induction. For the base case we want to understand the $p$-maximality of $f(x)=x^p-a$. Theorem \ref{Thm: SmithRadical} shows that $f(x)$ is $p$-maximal if and only if $a^p\not\equiv a \bmod p^2$. If $\ZZ[\alpha_1]$ is $p$-maximal, then since $x^p-a\equiv (x-a)^p \bmod p$, Dedekind-Kummer factorization shows that $p$ is totally ramified. 

For the induction step, take $n\geq 1$, and suppose that $\ZZ[\alpha_k]$ is $p$-maximal with $p$ totally ramified for all $k\leq n$. We have already seen that it is necessary that $v_p(a^p-a)=1$. We wish to show that this is sufficient for $\ZZ[\alpha_{n+1}]$ to be $p$-maximal with $p$ totally ramified. It suffices to show that $\Ocal_{K_n}[\alpha_{n+1}]$ is $\gp_n$-maximal over $\Ocal_{K_n}$ with $\gp_n$ totally ramified. Here $\gp_n$ is the unique prime of $\Ocal_{K_n}$ above $p$ and $p=(\gp_n)^{p^n}$. To this end, we apply the Montes algorithm to the polynomial $x^p-\alpha_n-a$, the minimal polynomial of $\alpha_{n+1}$ over $K_n$.

Note the residue field of $\gp_n$ is $\FF_p$. Hence, reducing $x^p-\alpha_n-a $ modulo $\gp_n$, we find
\[x^p-\alpha_n-a\equiv \left(x-\alpha_n-a\right)^p \bmod \gp_n.\]
If the constant coefficient of the $(x-\alpha_n-a)$-adic development of $x^p-\alpha_n-a$ has $\gp_n$-adic valuation 1, then the principal $(x-\alpha_n-a)$-polygon is one-sided with slope $-\frac{1}{p}$. This implies $\gp_n$ is totally ramified in $K_{n+1}$ and $\Ocal_{K_n}[\alpha_{n+1}]$ is $\gp_n$-maximal. We have
\[x^p=\left(x-\alpha_n-a+\alpha_n+a\right)^p=\sum_{i=0}^p \binom{p}{i} \left( x-\alpha_n-a \right)^{p-i}\left(\alpha_n+a\right)^i,\]
so
\begin{align*}
x^p-\alpha_n-a&=\sum_{i=0}^p \binom{p}{i} \left( x-\alpha_n-a \right)^{p-i}\left(\alpha_n+a\right)^i-\alpha_n-a\\
&=\sum_{i=0}^{p-1} \binom{p}{i} \left(\alpha_n+a\right)^i\left( x-\alpha_n-a \right)^{p-i}+\left(\alpha_n+a\right)^p-\alpha_n-a.
\end{align*}
From the binomial coefficients, we see the constant coefficient of the $(x-\alpha_n-a)$-adic development is 
\[\alpha_n^p-\alpha_n+a^p-a+ (\text{terms divisible by }p).\]
As $v_{\gp_n}(p)=p^n$, the terms divisible by $p$ are not relevant, and we are interested in the ${\gp_n}$-adic valuation of $\alpha_n^p-\alpha_n+a^p-a$. Since $p\mid(a^p-a)$ and $v_{\gp_n}(p)=p^n$, we want the ${\gp_n}$-adic valuation of $\alpha_n^p-\alpha_n$. We compute
\begin{align*}
v_{\gp_n}\left( \alpha_n^p-\alpha_n \right)&=v_{\gp_{n}}\left(\alpha_{n-1}+a-\alpha_n\right) & &\text{since } \alpha_n^p=\alpha_{n-1}+a, \\
&=v_{\gp_{n-1}}\left(\Norm_{K_n/K_{n-1}}\left(\alpha_{n-1}+a-\alpha_n\right)\right) & &\text{by Lemma \ref{Lem: Val Norms}}, \\
&=v_{\gp_{n-1}}\left(\prod_{i=1}^p\left(\alpha_{n-1}+a-\zeta_p^i\alpha_n\right)\right) & &\text{since } \alpha_{n-1},a\in K_{n-1}, \\
&=v_{\gp_{n-1}}\left(\left(\alpha_{n-1}+a\right)^p-\alpha_n^p\right) & &\text{using the factorization of } x^p-\alpha_{n-1}-a,\\
&=v_{\gp_{n-1}}\left(\left( \alpha_{n-1}+a \right)^p-\alpha_{n-1}-a\right) & &\text{since } \alpha_n^p=\alpha_{n-1}+a.
\end{align*} 
Note that when $n=1$, we have $K_0=\QQ$, $\gp_0=p$, and $\alpha_0=0$. In this case, the above shows 
\[v_{\gp_1}\left( \alpha_1^p-\alpha_1 \right)=v_p\left(a^p-a\right)=1.\] 
Now suppose $n>1$. Because $\Ocal_{K_{n-1}}[\alpha_{n}]$ is $\gp_{n-1}$-maximal, taking the $(x-\alpha_{n-1}-a)$-adic expansion of $x^p-\alpha_{n-1}-a$ shows 
\[v_{\gp_{n-1}}\big(\left( \alpha_{n-1}+a \right)^p-\alpha_{n-1}-a\big)=1.\]
Hence, $v_{\gp_n}\left(\alpha_n^p-\alpha\right)=1$. We see $\gp_n$ is totally ramified in $K_{n+1}$ and $\Ocal_{K_n}[\alpha_{n+1}]$ is $\gp_n$-maximal.
\end{proof}

Now we turn our attention to primes $\ell \neq p$. We roughly follow the strategy employed by \cite{Ruofan} and aim to establish the following lemma.

\begin{lemma}\label{Lem: PrimesAwayFromp}
Let $\ell$ be a prime not dividing $p$. 
\begin{enumerate}
\item\label{notmax} If $f^n(0)$ is divisible by $\ell^2$, then $\ZZ[\alpha_n]$ is not $\ell$-maximal.
\item\label{maximal} If $f^n(0)$ is not divisible by $\ell^2$, then $\Ocal_{K_{n-1}}[\alpha_n]$ is an $\gl$-maximal $\Ocal_{K_{n-1}}$-module for any prime ideal $\gl$ of $\Ocal_{K_{n-1}}$ above $\ell$.
\end{enumerate}
\end{lemma}

\begin{proof}
We begin with \eqref{notmax}. Note that $f^n(x)\in \ZZ[x^p]$. Hence, if $\ell\mid f^n(0)$, then 
\[f^n(x)\equiv x^p\prod \phi_i(x)^{e_i}\bmod \ell.\]
Taking the $x$-adic development, we see the initial vertex of the principal $x$-polygon is $\big(0,v_\ell(f^n(0))\big)$. Thus, if $\ell^2\mid f^n(0)$, then $\ZZ[\alpha_n]$ is not $\ell$-maximal.

We turn our attention to \eqref{maximal}. We have two cases. For $n=1$, Theorem \ref{Thm: SmithRadical} shows that $\ZZ[\sqrt[p]{a}]$ is $\ell$-maximal if and only if $\ell^2$ does not divide $a$. 

Now, suppose $n>1$ and let $\gl$ be a prime of $\Ocal_{K_{n-1}}$ above $\ell$. 
Note that $x^p-\alpha_{n-1}-a$ is the minimal polynomial of $\alpha_n$ over $\Ocal_{K_{n-1}}$. Since $\gl\nmid p$, Theorem \ref{Thm: SmithRadical} shows that $\Ocal_{K_{n-1}}[\alpha_n]$ is $\gl$-maximal if and only if $\gl^2$ does not divide $(\alpha_{n-1}+a)$. We will show the contrapositive of \eqref{maximal}, so suppose $\Ocal_{K_{n-1}}[\alpha_n]$ is not an $\gl$-maximal $\Ocal_{K_{n-1}}$-module. Thus $\gl^2\mid (\alpha_{n-1}+a)$. Since $\alpha_n^p=\alpha_{n-1}+a$, we have $\gl^2\mid \big(\alpha_n^p\big)$.
Taking norms,
\[\Norm_{K_{n}/\QQ}(\gl^2) \text{ divides } \Norm_{K_{n}/\QQ}\left(\alpha_n^p\right).\]
Therefore, 
\[\left(\Norm_{K_{n-1}/\QQ}(\gl^2)\right)^p \text{ divides } |f^n(0)|^p, \text{ and } \left(\Norm_{K_{n-1}/\QQ}(\gl)\right)^2 \text{ divides } |f^n(0)|.\]
Since $\ell\mid \Norm_{K_{n-1}/\QQ}(\gl)$, we see $\ell^2$ divides $f^n(0)$. Hence, if $\Ocal_{K_{n-1}}[\alpha_n]$ is not an $\gl$-maximal $\Ocal_{K_{n-1}}$-module for some prime ideal $\gl$ of $\Ocal_{K_{n-1}}$ above $\ell$, then $\ell^2\mid f^n(0)$. The contrapositive yields \eqref{maximal}.
\end{proof}

With our lemmas established, we combine them to get the main theorem.

\begin{proof}[Proof of Theorem \ref{Thm: Main}]
Lemma \ref{Lem: PrimesAwayFromp} shows that it is necessary that $\ell^2\nmid f^n(0)$ for every $n\leq N$ and for every prime $\ell\neq p$ in order to have $\ell$-maximality. Thus, suppose $\ell^2\nmid f^n(0)$ for all $n\leq N$. As Lemma \ref{Lem: pmaximal} already shows it is necessary that $a^p\not\equiv a\bmod p$, suppose that this also holds. 

We proceed by induction. The base case is clear and already present in the literature. Suppose $N>1$ and $\Ocal_{K_k}=\ZZ[\alpha_k]$ for all $k\leq n < N$. We wish to show that $\Ocal_{K_{n+1}}=\ZZ[\alpha_{n+1}]$. Lemma \ref{Lem: PrimesAwayFromp} shows that $\Ocal_{K_{n}}[\alpha_{n+1}]$ is $\gl$-maximal for every prime $\gl$ of $\Ocal_{K_n}$ with $\gl\nmid p$. Lemma \ref{Lem: pmaximal} shows that $\Ocal_{K_{n}}[\alpha_{n+1}]$ is $p$-maximal. Thus $\Ocal_{K_{n}}[\alpha_{n+1}]=\Ocal_{K_{n+1}}$. Since $\alpha_{n+1}^p=\alpha_n+a$ and because $\Ocal_{K_n}=\ZZ[\alpha_n]$, we see that $\ZZ[\alpha_{n+1}]=\Ocal_{K_{n+1}}$. With induction we have our result.
\end{proof}



\section{Examples}

The following examples are meant to help one get a flavor for the results above.

\begin{example}
Let $f(x)=x^3-5$. We have
\begin{align*}
f^1(x) &= x^3-5,\\
f^2(x) &= x^9 - 15 x^6 + 75 x^3 - 130,\\
f^3(x) &= x^{27} - 45 x^{24} + 900 x^{21} - 10515 x^{18} + 79200 x^{15} - 399375 x^{12} + 1350075 x^{9}\\ 
& \ \ \ - 2954250 x^{6} + 3802500 x^{3} - 2197005,\\
f^4(x) &= x^{81} - 135 x^{78} + 8775 x^{75} - 365670 x^{72} + 10974150 x^{69} - 252591750 x^{66} + 4636542150 x^{63}\\ 
& \ \ \ - 69676294500 x^{60} + 873198646875 x^{57} - 9248742526140 x^{54} + 83605735086975 x^{51}\\
& \ \ \ - 649601439751125 x^{48} + 4359787949171700 x^{45} - 25355305623690000 x^{42}\\ 
& \ \ \ + 127982660067337500 x^{39} - 560741779121461875 x^{36} + 2129668434875446875 x^{33}\\
& \ \ \ - 6990730404446390625 x^{30} + 19739990955501740700 x^{27} - 47622742788031456500 x^{24}\\
& \ \ \ + 97226962675508036250 x^{21} - 165792343518924835500 x^{18} + 231863628297715627500 x^{15}\\ 
& \ \ \ - 259125463888412765625 x^{12} + 222610926346013255625 x^{9} - 138078523258432818750 x^{6} \\
& \ \ \ + 55062074290560187500 x^{3} - 10604571775299775130.
\\
\end{align*}
We can see that explicitly computing the iterates of even a relatively simple radical polynomial quickly becomes unwieldy. However, from Theorem \ref{Thm: Main}, we know that $f^n(x)$ will always be 3-maximal since $9\dnd (125-5)$. Moreover, we need only check the factorization of the constant coefficients to see that $\ZZ[\sqrt[3]{5}]=\Ocal_{K_1}$, $\ZZ[\alpha_2]=\Ocal_{K_2}$, $\ZZ[\alpha_3]=\Ocal_{K_3}$, and $\ZZ[\alpha_4]=\Ocal_{K_4}$.
\end{example}

Consider the following for an example  where not every iterate supplies a monogenerator.

\begin{example}
Take $f(x)=x^5-5$. We have
\begin{align*}
f^1(x) &= x^5-5,\\
f^2(x) &= x^{25} - 25x^{20} + 250x^{15} - 1250x^{10} + 3125x^5 - 3130
,\\
f^3(x) &= x^{125} - 125x^{120} + 7500x^{115} - \cdots + 1499675775156250000x^5 - 300415051279300005,\\
f^4(x) &= x^{625} - 625x^{620} + 193750x^{615}- \cdots - (300415051279300005)^5-5.
\\
\end{align*}

Since $5$ is prime and $3130=2\cdot5\cdot 313$ is squarefree, we see $\Ocal_{K_1}=\ZZ[\sqrt{5}]$ and $\Ocal_{K_2}=\ZZ[\alpha_2]$. However $300415051279300005=3^{5} \cdot 5 \cdot 247255186238107$. Hence, Lemma \ref{Lem: PrimesAwayFromp} shows $\Ocal_{K_3}\neq \ZZ[\alpha_3]$. In particular, $\ZZ[\alpha_3]$ is not 3-maximal.  Thus we see that the first couple of iterates being monogenic does not ensure that every iterate will be monogenic. Interestingly, $f^4(0)$ is squarefree, so $\Ocal_{K_4}=\Ocal_{K_3}[\alpha_4]$. 
However, one can show $3$ divides $\big[\Ocal_{K_4}:\ZZ[\alpha_4]\big]$.
\end{example}

These examples motivate a question that was posed to the author by Caleb Springer.

\begin{question}\label{Q: AlwaysMono} Is there an $f(x)=x^p-a$ such that $f^n(x)$ is monogenic for all $n\geq 1$?
\end{question}
A polynomial such as $f(x)=x^2+1$ seems a promising candidate. Iterates of this polynomial are always $2$-maximal, and it can be shown that $f^n(0)$ is never divisible by $p^2$ for any $p<1000$. However, it appears that any answer to Question \ref{Q: AlwaysMono} would brush against difficult problems involving squarefree values of polynomials.

\section*{Acknowledgments} The author would like to thank Kimberly Ayers for the invitation that motivated this work. The author would also like to acknowledge Caleb Springer for kindly pointing out some oversights in an earlier draft.

\bibliography{Bibliography}
\bibliographystyle{alpha}

\end{document}